\documentclass[11pt,a4paper,reqno]{amsart}
\usepackage[utf8]{inputenc}
\usepackage{amsmath}
\usepackage{amsthm}
\usepackage{amsfonts}
\usepackage{amssymb}
\usepackage{graphicx}
\usepackage{tikz}
\pgfdeclarelayer{nodelayer}
\pgfdeclarelayer{edgelayer}
\pgfsetlayers{edgelayer,nodelayer,main}
\tikzset{newstyle/.style={thick}}
\tikzset{simple/.style={thick}}
\theoremstyle{plain}
\newtheorem{theorem}{Theorem}[section]
\newtheorem{lemma}{Lemma}[section]

\theoremstyle{definition}
\newtheorem{definition}{Definition}[section]

\theoremstyle{remark}
\newtheorem{remark}{Remark}[section]

\numberwithin{equation}{section}
\title[Maximal entropy dissipation numerical scheme]{Maximal entropy 
dissipation numerical scheme for conservation law
systems}
\author{Marko Nedeljkov}
\date{}
\address{University of Novi Sad, Faculty of Sciences, Department of Mathematics and Inofrmatics, Serbia}
\email{marko@dmi.uns.ac.rs}
\begin{document}
\maketitle	
\begin{abstract}
This paper presents a numerical finite volume method for conservation law 
systems that are adapted to the principle of maximal dissipation. 
The general assumptions are an existence
of a strictly convex entropy functional and a finite propagation speed 
property of a given system. The procedure is based on a numerical flux
construction obtained by the minimization of the entropy functional
in each time step. The scheme satisfies assumptions
of the Lax-Wendroff theorem.  
A limiting solution obtained by this scheme is compared with the classical
weak solutions obtained by the Glimm or the Wave Front Tracking algorithm
for one-dimensional systems, respectively. 
\end{abstract}
\section{Introduction}
The principle of maximal entropy dissipation was introduced in \cite{CD_1973} 
as a tool for extracting physically relevant weak solutions to 
conservation law system. We are using more recent versions from 
\cite{CD_2012} and \cite{F2014}. The paper \cite{Hsiao} contains an early
comparison between the Lax entropy principle and proves several important differences. 
Suppose that we have a fixed mathematical 
entropy functional, together with the conservation law system. Using
this principle one can extract a weak solution that maximally dissipates 
the mathematical entropy in an arbitrary short time period.
The verification of this could be quite demanding for weak solutions.
Naturally, the term ``arbitrarily  small time step'' associates us to some  
a numerical procedure. Following that intuitive thought, we will make 
a numerical scheme that satisfies 
principle at each time step. A natural choice for the scheme is
Finite Volume Method (FVM). It is a link between the weak solution concept with test 
functions and the integral form, which is the primary tool for deriving
equations of fluid dynamics, see the classical monographs
\cite{Batchelor} or \cite{MarsdenChorin}. 
\medskip

The main properties and assumptions of this procedure are as follows:
\medskip 

\noindent -- The finite propagation property speed (FPPS) is crucial, 
and we localize the problem
by adding ghost cells for a given finite time interval, 
where a solution exists. A lot
of technical details are standard procedures from the books 
\cite{GR} and \cite{leveque}. For simplicity,
we are looking for solutions constant out of a compact set until the time 
variable reaches a fixed $T$. That can be relaxed by adding 
boundary data and an adequate choice
of values in the ghost cells, like in the books cited above.
However, this makes the estimation of entropy dissipation more difficult.
\medskip

\noindent -- The scheme is made to be conservative and consistent and one may use
the Lax-Wendroff theorem. 
\medskip 

\noindent -- We choose a numerical flux function adapted to the strictly convex
entropy. This provides a unique numerical solution at each time step.
\medskip

\noindent -- Any affine constraint, like non-negativity of the fluid 
physical density and energy can be easily added in the algorithm without 
losing the uniqueness at each time step.

\noindent -- The definition and procedure for deriving a solution in 
each time step are dimension-independent. The numerical flux depends 
on values in each neighborhood cell and a vector with components 
bounded between 0 and 1 with dimension equals the number of all
cell-to-cell surface. That makes the optimization problem quite challenging 
despite one knows that there is a unique solution due to the strict convexity
of the entropy functional and affine constraints. 
A na\'{i}ve, direct approach works reasonably fast only in one spatial dimension. 
For multidimensional problems, one has to use some more sophisticated numerical 
optimization technique. Recent development in high-performance computing 
gives a chance for this procedure to be useful. 
\medskip 

\noindent -- Due to the well-developed theory in one space dimension, we are able
to compare a limiting solution obtained by the procedure with a weak 
solution obtained by the Glimm or the Wave Front Tracking algorithm 
(WFT), which we shall call "classical" in the sequel.
Two main reasons are a bounded variation property of these solutions and the fact
that singularities lie in a set with zero Lebesgue measure. 
One can look in \cite{Diperna} (Glimm scheme) and \cite{Bressan}, 
\cite{Holden-Risebro} (WFT) for the construction and 
properties of these solutions.
\medskip 

\noindent -- A direct comparison in a multi-dimensional case was impossible 
by the methods used in this paper because of the lack of the above classical 
solutions properties.  If a bounded multi-dimensional
solution has a zero measure singularity set and has bounded variation 
in each direction for a fixed time variable, 
then the comparison could be possible 
by extending a one-dimensional proof. Some of the  interesting problems
are described in \cite{DLS2010} and \cite{CK2014}.
\medskip

Consider an $n$-dimensional system of hyperbolic conservation laws 
\begin{equation} \label{claws}
U_{t}+\nabla \cdot F(U) = 0 \text{ in } \mathbb{R}_{+}^{m}:=\mathbb{R} ^{m}
\times  (0,\infty ), \;  U(x,0)=U_{0}(x), 
\end{equation}
where the initial data $U_{0}$ is constant out of some compact set 
in $D_{0}\subset \mathbb{R}^{m}$.
Denote by $B_{T}$ a closed subset of $\mathbb{R}^{n}\times [0,T]$, for some $T>0$,
containing the range of $U$. 
This will be referred to as the domain in the sequel. 
Assume $F\in C^{1}(B_{T})$ and that there exists at least one  
mathematical entropy strictly convex function $\eta \in C^{2}(B_{T})$.
The crucial assumption is that we can choose $B_{T}$ such that 
the FPSP holds. A typical case is 
when $B_{T}$ is bounded because the existence of convex $\eta$ implies 
the FPSP. The ideal situation is when $\eta$ represents some physical quantity
related to the system like the physical entropy with a reversed sign, or energy, 
for example. We fix the function $\eta$ together with the system (\ref{claws})
in the rest of the paper, together with the Courant--Friedrichs--Lewy constant 
\[
\begin{split}
C_{\rm{cfl}}=C_{\rm{cfl}}(B_{T}) < \min\Big\{ & 1,   
\inf_{\omega \in S^{m-1}}\{(|\lambda_{i}(\omega,U)|)^{-1}, \;  \lambda_{i}
\text{ are the eigenvalues} \\ & \text{of system (\ref{claws})},\;   i=1,\ldots,m\}, \; U\in B_{T} \Big\}.
\end{split}
\]
Then we are sure that $U$ is a constant out of the set 
$$
D_{T}=\{(x\pm \omega \frac{t}{C_{\rm cfl}},\; \omega \in S^{m-1},\; x\in D_{0},\; 
t\in[0,T]\}. 
$$
 
A vector valued function $U\in (L_{\rm{loc}}^{1}(\mathbb{R}_{+}^{d}))^{n}$ 
is said to be a weak 
solution to the system (\ref{claws}) if for every function 
$\phi \in C_{0}^{\infty}(\mathbb{R}^{d}\times (-\infty,\infty ))$ 
we have
\begin{equation} \label{sol}
\int^{\infty }_{0}\int_{\mathbb{R} ^{d}}\left( U\varphi_{t}+F\left( U\right) 
\cdot \nabla \varphi \right) \,dx \,dt+\int_{\mathbb{R} ^{d}}U_{0}
\left( x\right) \varphi \left( x,0\right) \,dx=0.
\end{equation}

The main goal of this paper is to find a numerical procedure for finding a weak
solution to ({\ref{claws}) that minimizes an infinitesimal increase of the functional  
\[
\mathcal{E}(U)(t)=\int_{(x,t)\in D_{T}}\eta(U(x,t))\, dx
\] 
in the sense of Dafermos (\cite{CD_1973}). More precisely, we use a slightly different condition from \cite{F2014}.
 
\begin{definition} \label{def-dafermos}
A weak solution $U$ to system (\ref{claws}) satisfies the energy
admissibility condition if for any other weak solution
$\bar{U}$ such that for some $\tau\in [0,T)$ we have
$U(\cdot,t)=\bar{U}(\cdot,t)$, $t\leq \tau$ and 
there exists a sequence $\{\tau_{i} \}_{i \in \mathbb{N}}$, $\tau_{i}>\tau$,
$\tau_{i}\to \tau$, as $n\to \infty$ such that
\[
\begin{split}
& \mathcal{E}(\bar{U})(\tau +)\geq \mathcal{E}(U)(\tau +), \\ 
& \text{where }
\mathcal{E}(U)(\tau +)=\mathop{\rm esslim}_{t \to \tau, t>\tau}\mathcal{E}(U)(t)
\text{ in the distributional sense}.
\end{split}
\]
\end{definition}

\section{Approximate solution}

In this section, we shall define the specific FVM algorithm
and use the Lax-Wendorff theorem. 
The notation and assertions used here are from 
\cite{jegdic}.
Consider the following  partition of $S\subset \mathbb{R}_{+}^{m}$:
\[
\Delta:=\{ \Omega_{i}\times [ t_{j},t_{j+1}) :\; i\in I,\; j\in \mathbb{N}_{0}\},
\; S=\bigcup_{i\in I}\Omega _{i},
\]
where $I$ is a  finite set of indices. The procedure  
goes as follows. 
Define 
$U_{i}^{0}=\dfrac{1}{|\Omega_{i}|} \int_{\Omega_{i}} U_{0}(x) \, dx$, $i\in I$.
Put $U_{0\Delta}(x):=U_{i}^{0},\; x\in \Omega_{i}$ for the initial data, and 
let $U_{\Delta}(x,t):=U_{i}^{j} \; x\in \Omega_{i},\; t\in [t_{j},t_{j+1})$ 
is a piecewise constant approximation of the solution. 
The values $(U_{i}^{j})_{i\in I}$ for $j=1,2\ldots$ are updated according to 
\begin{equation} \label{shceme}
U_{i}^{j+1}=U_{i}^{j}-\dfrac{\Delta t}{|\Omega _{i}| }
\sum_{k\in K_{i}}\int_{S_{i,k}}h_{F\cdot \nu_{i,k}}(U_{i}^{j},U_{k}^{j})\, dS,
\end{equation}
where $K_{i}$ denotes the set of cell indices $k$ with $\Omega_{k}$ adjacent to 
$\Omega_{i}$, $|\Omega_{i}|$ is the volume of cell $\Omega_{i}$, $S_{i,k}$ is 
the surface between cells $\Omega_{i}$ and $\Omega_{k}$, $\nu_{i,k}$ 
is the outward unit normal to $\Omega_{i}$ along $S_{i,k}$, 
and $\Delta t=t_{j+1}-t_{j}$.

\begin{remark}
The simplest choice is equidistant rectangular mesh
$\Omega_{i}=[x_{i_{1}},x_{i_{1}+1}:=x_{i_{1}}+\Delta x]\times \ldots \times
[x_{i_{d}},x_{i_{d}+1}:=x_{i_{d}}+\Delta x]$, 
with the condition $\dfrac{\Delta t}{\Delta x}\leq C_{\rm{cfl}}$,
and $\nu_{i,k}$ being constant along the side $S_{i,k}$. Thus,
\[
\int_{S_{i,k}}h_{F\cdot \nu_{i,k}}\left( U_{i},U_{k}\right) = 
|S_{i,k}| h_{F\cdot n_{i,k}}\left( U_{i},U_{k}\right). 
\]
That should suffice for
regular enough domains. One can always use a unstructued FVM mesh with
some more calculations for $\nu_{i,k}$.
\end{remark}

Let $\eta$ be a fixed strictly convex entropy function for 
(\ref{claws}), regular enough in $B_{T}$, and let $D_{T}$ 
be bounded region such that $U$ is constant out of it for a finite
$T$ as discussed above. We are looking for a piecewise
constant approximation $U_{\Delta}:D_{T}\to B_{T}$ of a weak solution 
assuming that $D_{T}$ is big enough for physically relevant 
solutions with the given initial data. 
This requires adding ghost cells in numerical calculations, 
see \cite{GR} or \cite{leveque}.
\begin{figure}[h] 
	\label{fig1}
\begin{tikzpicture}[scale=0.8]
\begin{pgfonlayer}{nodelayer}
\node (0) at (-10, 21) {};
\node  (1) at (-6, 21) {};
\node  (2) at (-12, 21) {};
\node  (3) at (-4, 21) {};
\node  (4) at (-14, 21) {};
\node  (5) at (-2, 21) {};
\node  (6) at (-14, 26) {};
\node  (7) at (-12, 26) {};
\node  (8) at (-4, 26) {};
\node  (9) at (-2, 26) {};
\node  (10) at (-12, 20.5) {$x=-L$};
\node  (11) at (-4, 20.5) {$x=L$};
\node  (12) at (-10, 20) {};
\node  (13) at (-6, 20) {};
\node  (14) at (-15, 26) {};
\node  (15) at (-15, 26) {$t=T$};
\node  (16) at (-14, 24) {$U=U_{l}$};
\node  (17) at (-2, 24) {$U=U_{r}$};
\node  (18) at (-8.25, 24.5) {};
\node  (19) at (-8.25, 24.5) {$D_{T}$ (black and red border lines)};
\node  (20) at (-11.25, 22) {$U=U_{l}$};
\node  (21) at (-4.75, 22) {$U=U_{r}$};
\node  (22) at (-10, 19.5) {$x=-L+T/C_{\rm{cfl}}$};
\node  (23) at (-6, 19.5) {$x=L-T/C_{\rm{cfl}}$};
\end{pgfonlayer}
\begin{pgfonlayer}{edgelayer}
\draw [very thick, red] (0.center) to (1.center);
\draw [very thick, black] (2.center) to (0.center);
\draw [very thick, black] (1.center) to (3.center);
\draw [very thin, gray] (4.center) to (2.center);
\draw [very thin, gray] (3.center) to (5.center);
\draw [very thin,gray] (6.center) to (7.center);
\draw [very thin,gray] (8.center) to (9.center);
\draw [style=blue] (1.center) to (8.center);
\draw [style=blue] (0.center) to (7.center);
\draw [very thick, red] (7.center) to (8.center);
\draw [very thick, black] (3.center) to (8.center);
\draw [very thick, black] (2.center) to (7.center);
\draw [thin, ->] (12.center) to (0.center);
\draw [thin, ->] (13.center) to (1.center);
\end{pgfonlayer}
\end{tikzpicture}
\vspace*{-0.4cm}
\caption{\small The region $D_{T}$ in 1-dim case. Non-constant region is 
between the red and blue lines. The ghost cells are in the regions bounded by black and blue lines.}
\end{figure}
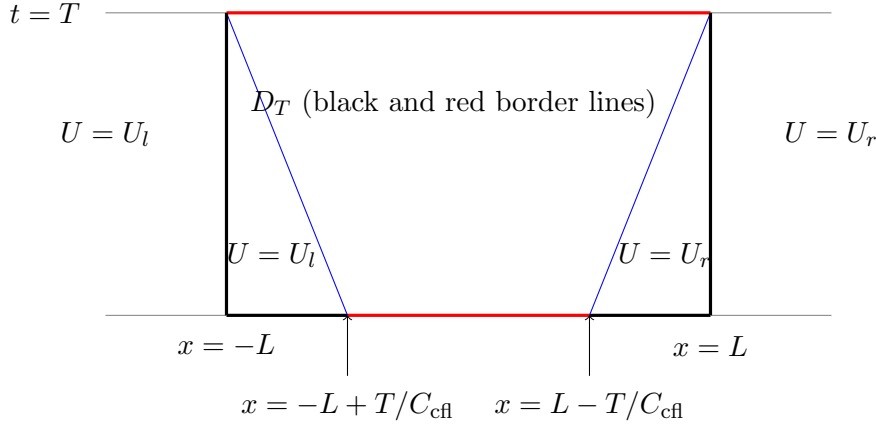
\begin{definition} \label{defscheme}
For a given initial data $U_{0\Delta}$ we construct approximate 
(i.e.\ numerically obtained) solution 
$U_{\Delta}$ in $D_{T}$ by (\ref{shceme}) where
\begin{equation} \label{nflux}
h_{F\cdot \nu_{i,k}}(U_{i}^{j},U_{k}^{j})
=(\theta_{i,k}^{j} \cdot \underline{F}_{i,k}^{j}+(1-\theta_{i,k}^{j})\cdot
\overline{F}_{i,k}^{j})\cdot \nu_{i,k}, 
\end{equation}
where 
\[
\begin{split}
\underline{F}_{i,k}^{j}&:=\min\{F(U):\; U\in [\min\{ U_{i}^{j},
U_{k}^{j}\},\max\{U_{i}^{j}, U_{k}^{j}\}]\}\\
\overline{F}_{i,k}^{j}&:=\max\{F(U):\; U\in [\min\{ U_{i}^{j},
U_{k}^{j}\},\max\{U_{i}^{j}, U_{k}^{j}\}]\},
\\ & \theta_{i,k}^{j}=\theta_{k,i}^{j}.
\end{split}
\] 	
For each $j\in \mathbb{N}_{0}$, $i,k\in I$, 
$\theta_{i,k}^{j}$ is $n$-dimensional vector with values in $[0,1]$, 
$1$ denotes the $n$-dimensional vector with each component equals $1$, 
and ``$\cdot$'' is the usual scalar product. 
The functions $\min$ and $\max$ are used component-wise.
	
The elements $\{\theta_{i,k}^{j}\}_{i,k\in I}\in [0,1]^{n}$ are chosen such that 
\[\eta^{j+1}(U_{\Delta}):=\int_{\Omega_{T}}\eta(U_{\Delta}(x,t_{j+1}))\,dx\] 
is minimal for each time step $j$.
\end{definition}
\begin{remark}

$U_{i}^{j+1}$ is the affine function of $\Theta^{j}
=\big( \theta_{i,k}^{j}\big)_{i,k\in I}$ (rearranged to be a vector), so
\[
\eta^{j+1}(U_{\Delta})=\int_{\Omega_{T}}\eta(U_{\Delta}(x,t_{j+1}))\, dx
=\dfrac{1}{|\Omega_{i}|}\sum_{i\in I} \int_{\Omega_{i}}\eta(U_{i}^{j+1})\, dx
\]
is convex with respect to $\Theta^{j}$. For every $j$, 
we have a unique minimum, but may
be obtained by multiple vectors $\Theta^{j}$. However, the restriction of a 
strictly convex function to an affine subspace is still strictly convex. 
A strictly convex function can have at most one minimizer in its domain.
Therefore, if a minimum is attained, the vector $U_{i}^{j+1}$ that achieves it 
is unique, although infinitely many $\Theta^{j}$ may map to it.
Thus, scheme (\ref{shceme}) is well defined providing that 
$\big(U_{i}^{j+1}\big)_{i\in I}$ belongs to $B_{T}$ and satisfies the affine constraints.

If we have some affine constraints on the variables, like non-negativity of a solution component,
this also creates affine constraints on $\Theta ^{j}$ in the optimization algorithm. Therefore, uniqueness is preserved.
\end{remark}

\begin{remark} \label{rem2}
The scheme is chosen such that for each $U_{i,k}^{j}$ lying componentwisely
between 
$U_{i}^{j}$ and $U_{k}^{j}$ ($U_{i,k}^{j}\in [\min\{ U_{i}^{j},
U_{k}^{j}\},\max\{U_{i}^{j},U_{k}^{j}\}]$)
for every $i,k\in I$, $F(U_{i,k})\cdot \nu_{i,k}=h_{F\cdot \nu_{i,k}}(U_{i}^{j},U_{k}^{j})$ for some $\theta_{i,k}^{j} \in [0,1]^{n}$.
We will use this fact to compare the solution from Theorem \ref{lw} bellow
with the classical solutions in the one-dimensional case obtained by 
the Glimm or WFT scheme. 

Note that the opposite is not true: it is not always possible to find $U_{i,k}$
described above for each choice of $(\theta_{i,k}^{j})_{i,k\in I} \in [0,1]^{n}$.

If one wants to have that option, one possibility is to define
\[
h_{F \cdot \nu_{i,k}} (U_{i}^{j}, U_{k}^{j}) = F(\theta _{i, k}^{j} U_{i}^{j} + (1 - \theta _{i, k}^{j}) U_{k}^{j}.
\]
But then $\eta^{j+1} (U _{\Delta})$ would not be convex with respect to
$\Theta ^{j}$. Then a minimum is maybe non-unique and a numerical 
procedure for finding the global minimum does not exist in general.

If $U_{i}^{j}=U_{k}^{j}$, then $\theta_{i,k}^{j}$ can by anything, that is, 
$F(U_{i}^{j})\cdot \nu_{i,k}=F(U_{i}^{j})\cdot \nu_{i,k}=h_{F\cdot 
\nu_{i,k}}(U_{i}^{j},U_{k}^{j})$. (Then, $U_{i,k}=U_{i}^{j}=U_{k}^{j}$.)
\end{remark}

Taking all that into account, one can see that the proposed scheme,
called MES scheme in the sequel, is consistent 
and conservative. In addition, $h_{F\cdot \nu_{i,k}}$ is globally Lipschitz in 
both $U_{i}^{j}$ and $U_{k}^{j}$ in the domain. 
All that means that the assumptions of the Lax-Wendorff Theorem are satisfied, 
so we have 
\begin{theorem} \label{lw}
Let $U_{\Delta}:D_{T} \to B_{T}$, assume $B_{T}$ is bounded and 
$\eta\in \mathcal{C}^{2}(B_{T})$.
If $U_{\Delta} \to U$ boundedly almost everywhere on $D_{T}$ as $|\Delta|\to 0$, then the limit function $U$ is a weak solution of (\ref{claws}).
\end{theorem}

The relation $|\Delta|\to 0$ means $\Delta x \to 0$ with 
$\Delta t < \Delta x \cdot C_{\rm{cfl}}$.
Let us call the limiting function $U$ from Theorem \ref{lw} the 
Minimal Entropy Solution (MES solution).
Note that the boundedness of $B_{T}$ implies that there exists a subsequence 
of $U_{\Delta}$ converging in the space of signed Radon measures $\mathcal{M}(D_{T})$.  However, this limit may not satisfy (\ref{sol}) in the usual sense.

\subsection{Equivalence of weak solution and integral representation definitions}
	
In 1D case, $U\in L^{1}_{loc}(\mathbb{R}_{+}\times \mathbb{R})$ is a weak solution 
(sometimes called ``distributional'') if for every test function 
$\varphi \in C^{\infty}(\mathbb{R}\times \mathbb{R})$, (\ref{sol}) holds true.
This is equivalent to the following relation (see \cite{CM} or \cite{CD2005}):

\noindent
{\bf The integral form} in 1D case says that 
\begin{equation} \label{if}
\int ^{x_{2}}_{x_{1}}U( x,t_{2}) -U( x,t_{1}) \, dx+\int ^{t_{2}}_{t_{1}}
F( U( x_{2},t) ) -F( U( x_{1},t) ) \, dt=0,
\end{equation}
for every rectangle $[x_{1},x_{2}] \times [t_{1},t_{2}]$.
In the multi-dimensional case
\[
\int _{W}U(x,t_{2}) dV-\int _{W}U(x,t_{1}) dV=-\int_{t_{1}}^{t_{2}}\int _{\partial W}{F}( U( x,t) ) \cdot \nu \, dS\, dt,
\]
has to be true for every $0\leq t_{1} < t_{2}$ and $W$ being any 
$m$-dim rectangle or some other volume element with 
a sufficiently regular boundary.

One can look at the first chapter of the book \cite{CD2005} for much more 
detailed analysis of the mathematical foundations of balance laws.

\section{MES in 1D}
Let us fix $L>0$ such that $D_{T}\subset [-L,L]\times [0,T]$. 
Suppose $\widehat{U}$ is a bounded classical weak solution to (\ref{claws}) 
for $m=1$ of finite bounded variation
for each $t$ (obtained by the Glimm or WFT scheme as we have written above), 
$\widehat{U}\in L^{\infty}(D_{T}) \cap BV (D_{t})$, 
$D_{t}=\{(x,t)\in D_{T}, \; t\in [0,T)\}$.
For both of them,  $\mathop{\rm TV}(\widehat{U}(x,t))\leq {\rm const}
\mathop{\rm TV}(U_{0})$. 

The scheme from Definition \ref{defscheme} can be written in the simpler form
for 1D case with $\Omega _{i} = \left[ x _{i - \frac{1}{2}} , x _{i + \frac{1}{2}} \right]$,
\begin{equation} \label{1dscheme}
\begin{split}
U_{i}^{j+1} =  & U_{i}^{j}+\Big( \theta_{i}^{j}\underline{F}_{i-1/2}^{j}
+(1-\theta_{i}^{j})\overline{F}_{i-1/2}^{j}) \\
& -\big(\theta_{i+1}^{j}\underline{F}_{i+1/2}^{j}
+(1-\theta_{i+1}^{j})\overline{F}_{i+1/2}^{j}\big) \Big) \frac{\Delta t}{\Delta x}\\
& \underline{F}_{i-1/2}^{j}=\min\{F(U):\; 
U\in [\min\{U_{i-1},U_{i}\},\max\{U_{i-1},U_{i}\}]\} \\
& \overline{F}_{i-1/2}^{j}=\max\{F(U):\; 
U\in [\min\{U_{i-1},U_{i}\},\max\{U_{i-1},U_{i}\}]\},\\ 
& i\in I,\; j\in \mathbb{N}.
\end{split}
\end{equation}
All minimums, maximums, and relation $\in$ are taken componentwisely, while $[\cdot,\cdot]$
denotes an $n$-dimensional interval.

In the following assertion, we will show how to approximate a classical 
solution using the MES scheme.

\begin{lemma} \label{lemma3.1}
Let $\widetilde{U}$ be the piecewise constant approximation of the classical weak solution
$\widehat{U}$ defined on the equidistant mesh 
$\{ (x_{i-1/2},t_{j})\}_{i\in I, j\in \mathbb{N}}$, $\Delta x=x_{i+1/2}-x_{i-1/2}$, 
$\Delta t=t_{j+1}-t_{j}$ using (\ref{if}):
\[
\begin{split}
& \widetilde{U}(x,t)=\widehat{U}_{i}^{j},\; (x,t) \in (x_{i-1/2},x_{i+1/2})\times [t_{j},t_{j+1}) \\
& F(\widetilde{U}(x_{i+1/2},t))=\widehat{F}_{i+1/2,j},\; t\in [t_{j},t_{j+1}). 
\end{split} 
\]
where 
\begin{equation}\label{kappa} 
\begin{split}
& \widehat{U}_{i}^{j}:= \frac{1}{x_{i+1}-x_{i}}\int_{x_{i-1/2}}^{x_{i+1/2}} \widehat{U}(x,t_{j}) \, dx \\
& \widehat{F}_{i+1/2}^{j}:=\frac{1}{t_{j+1}-t_{j}}\int ^{t_{j+1}}_{t_{j}}F(\widehat{U}(x_{i+1/2},t)) \, dt.
\end{split}
\end{equation}

Then for every $j\in \mathbb{N}_{0}$ 
there exists a fine enough mesh and a FVM approximate solution 
$\breve{U}$ defined at the same mesh and satisfying (\ref{1dscheme}) such that 
\begin{equation}\label{s}
\sum_{i\in I}|\eta(\breve{U}_{i}^{j+1})-\eta(\widetilde{U}_{i}^{j+1})|\Delta x
\end{equation}
is arbitrary small.
\end{lemma}

\begin{proof}
The piecewise constant function $\widetilde{U}$ converge to 
the solution $\widehat{U}$ in 
distributional sense (in $L^{1}(D_{T})$ also) since $B_{T}$ 
is bounded by the definition of classical solutions. 

For every $t\in[0,T)$, a bounded variation function is
differentiable almost everywhere (see \cite{RoFi}, for example). 
Thus, the function $\widehat{U}$ has countably many discontinuities
for each line $t=t_{j}$. Also it has at most countably many discontinuity lines
as one can see in \cite{Diperna}
for the Glimm and \cite{Bressan} for the WFT scheme.
That is, the union of discontinuity lines may be covered by a subset of rectangles
in the union $\bigcup_{i,j \in \mathbb{N}} [x_{i-1/2},x_{i+1/2}] 
\times [t_{j},t_{j+1}]$, where 
each rectangle has arbitrary small volume 
$\Delta x \cdot \Delta t$. 
Denote by $I_{\rm nd}^{j}\subset I$ a set of indices such that  
$k\in I_{\rm nd}^{j}$ if $\Omega_{k}$ contains a discontinuity. Since a set of discontinuities
is at most countable, its Lebesgue measure equals zero, 
and there is a fine enough mesh such that 
$\sum_{k\in I_{\rm nd}^{j}}|\Omega_{k}| < \varepsilon$, 
for every  $\varepsilon>0$. 

In the strip $D_{[t_{j},t_{j+1}]}:=\{ (x,t)\in D_{T}:\; t\in [t_{j},t_{j+1}]\}$ 
the function $\widehat{U}$ is continuously
differentiable out of the cells with indices in $I_{\rm nd}^{j}$ and
$I_{\rm nd}^{j+1}$. 

Finally, let us denote by $I_{\rm m}^{j}\subset I \setminus I_{\rm nd}^{j} 
\setminus I_{\rm nd}^{j+1}$ a set of indices 
$i\in I$ such that some component of $\widehat{U}(\cdot,t_{j})$ reaches 
a maximum or minimum in the cell  
$[x_{i-1/2},x_{i+1/2}]$ at $t=t_{j}$, 
i.e.\ each component of $\widehat{U}(\cdot,t_{j})$ is monotone 
for $i \not\in I_{m}$. Again, this set $I_{m}^{j}$ 
of cells with extreme points of $\widehat{U}(\cdot,t)$ for each $t$ 
is at most countable due to the BV property of the solution in one space 
dimension. Thus, one can refine the above mesh such that
$\sum_{k\in I_{\rm m}^{j}}|\Omega_{k}| < \varepsilon$.

We shall construct a new approximate solution for $t=t_{j+1}$ 
and for all the cells. 
The non-negligible cells are the ones with indices in 
$I_{0}:=I \setminus I_{\rm nd}^{j} \setminus I_{\rm nd}^{j+1} 
\setminus I_{\rm m}^{j}$ where all the components on the solution are 
smooth and monotone.
We have to estimate the difference in the entropy integral
between values in $\widehat{U}(\cdot,t_{j+1})$ 
and new ones only in these cells.

Since $\hat{U}$ is smooth
in each $\Omega_{i}$, $i\in I_{0}$, and $F\in \mathcal{C}(B_{T})$ one can find value 
$\tau\in [t_{j},t_{j+1}]$ such that
$$\hat{F}_{i-1/2}^{j}=
\frac{1}{\Delta t} \int_{t_{j}}^{t_{j+1}}F(\widetilde{U}(x_{i-1/2},t))dt 
= F(\widetilde{U}(x_{i-1/2},\tau)).$$
Put
$\breve{F}_{i-1/2}^{j}:=F(\widetilde{U}(x_{i-1/2},t_{j}))\in 
[\underline{F}_{i-1/2}^{j},\overline{F}_{i-1/2}^{j}]$,
such that (\ref{1dscheme}) is satisfied: 
We have seen in Remark \ref{rem2}, that there exists a vector
$\theta_{i}^{j}\in [0,1]^{n}$ such that
\[
\breve{F}_{i-1/2}^{j}:=(\theta_{i}^{j}\underline{F}_{i-1/2}^{j}
+(1-\theta_{i}^{j})\overline{F}_{i-1/2}^{j}).
\]
The notation from (\ref{1dscheme}) means that
\[ 
\underline{F}_{i-1/2}^{j}=\min_{x\in[x_{i-1/2},x_{i+1/2}]} F(\widetilde{U}(x,t_{j})) \text{ and }
\overline{F}_{i-1/2}^{j}=\max_{x\in[x_{i-1/2},x_{i+1/2}]} F(\widetilde{U}(x,t_{j})).
\]
Each component of $\widetilde{U}(\cdot,t_{j})$ is continuous and monotone
in $[x_{i-1/2},x_{i+1/2}]$, $i\in I_{0}$, and 
$$
\Big| \hat{F}_{i-1/2}^{j} - \breve{F}_{i+1/2}^{j} \Big|
\leq \underbrace{\| \nabla_{x} F\|_{L^{\infty}}}_{=:C_{F}^{1}}(t_{j}-\tau)
\leq C_{F}^{1} \Delta t.
$$
Using the integral representation (\ref{if}) and (\ref{kappa}), we have
\[ 
\widehat{U}_{i}^{j+1}=\widehat{U}_{i}^{j}+\big( \widehat{F}_{i-1/2}^{j}-\widehat{F}_{i+1/2}^{j} \big)
\frac{\Delta t}{\Delta x}
\]
for each $i\in I$. Put
\[ 
\breve{U}_{i}^{j+1}=\widehat{U}_{i}^{j}+\big( \breve{F}_{i-1/2}^{j}-\breve{F}_{i+1/2}^{j} \big)
\frac{\Delta t}{\Delta x},
\]
Then,
\[
I:= \Big| \sum_{i\in I_{0}}\eta(\breve{U}_{i}^{j+1})\Delta x
-\sum_{i\in I_{0}}\eta(\widehat{U}_{i}^{j+1})\Delta x
\Big| \leq C_{\eta}^{1} \sum_{i\in I_{0}} 
| \breve{U}_{i}^{j+1} - \widehat{U}_{i}^{j+1} | \Delta x,
\]
where $C_{\eta}^{1}:=\sup_{U\in B_{T}}\|D\eta (U)\|_{L^{\infty}}$.
Substituting the above values for $\breve{U}_{i}^{j+1}$ and $\widehat{U}_{i}^{j+1}$, we have
\[
| \breve{U}_{i}^{j+1} - \widehat{U}_{i}^{j+1} | \leq  
\big( \big|\widehat{F}_{i+1/2}^{j}-\breve{F}_{i+1/2}^{j}\big| 
+\big|\widehat{F}_{i-1/2}^{j}-\breve{F}_{i-1/2}^{j}\big|\big) 
\frac{\Delta t}{\Delta x} \leq 2\frac{\Delta t}{\Delta x}C_{F}^{1}\Delta t.
\]
Assuming that $\frac{\Delta t}{\Delta x} \leq C_{\rm cfl}<1$, 
\[
I \leq 2C_{\eta}^{1} C_{F}^{1} \Big( \sum_{i\in I_{0}}\Delta x\Big) 
\Delta t \leq 4 L C_{\eta}^{1} C_{F}^{1} \Delta t. 
\] 
Thus, $I<\varepsilon$ for $\Delta t$ small enough.
The sum (\ref{s}) stays $\mathcal{O}(\varepsilon)$ since 
$\eta$ is regular enough and 
$\sum_{i\in I_{\rm nd}^{j} \cup I_{\rm nd}^{j-1} \cup I_{\rm m}^{j}}
|\Omega_{i}|\leq \mathop{\rm const} \varepsilon$.
Finally, note that $\Delta t$ does not depend on $j\in \mathbb{N}_{0}$, 
due to the assumption that the flux, entropy functions are regular enough,
and solutions take their values in an appropriate set $B_{T}$.
\end{proof}

\begin{remark}
With the above assumptions and notation, we prove the following additional property:
For every $\varepsilon$ and a fixed $C_{\rm cfl}$, there exists $\mu$ such
that for $\Delta t<\mu$ we have
\begin{equation} \label{hatlinf}
\sum_{i\in I\setminus I_{\rm nd}^{j}\setminus I_{\rm nd}^{j+1}
\setminus I_{\rm m}^{j}}
\| \widehat{U}(x,t_{j}) - \widehat{U}_{i}^{j}\|_{L^{\infty}(D_{T})}{\Delta x} 
< \varepsilon.
\end{equation}
The above relation means that the approximation of a weak 
solution by excluding all ``irregular'' cells does not change 
t significantly. Therefore, the MES satisfies the Dafermos condition 
(\ref{def-dafermos}), as shown in the following theorem.
\end{remark}

\begin{remark} \label{rem4}
The numerical flux property from Remark \ref{rem2} is crucial for this proof.
It is used to define $\breve{F}_{i-1/2}$ in the above proof.
Note that the simplest choice $F_{i,k}^{j} = F \big( U _{i}^{j} \big)$,
$\overline{F} _{i , k} = F \big( U _{k}^{j} \big)$, which gives
\[
\begin{split}
U_{i}^{j+1} =  & U _{i}^{j} + \Big( F(U_{i-1}^{j})
\theta _{i}^{j} + F( U_{i}^{j}) (1 - \theta _{i}^{j}) \\ 
& - \big( F( U _{i}^{j}) \theta_{i + 1}^{j} + F( U_{i + 1}^{j})
( 1 - \theta _{i + 1}^{j} ) \big) \Big) \frac{\Delta t}{\Delta x}
\end{split}
\]
for (\ref{1dscheme}) in the 1D case does not possesses that property.
The simple example with $U=(u,v)$ where $f(u,v)$ denotes
a component of the flux is sufficient to show that. Take 
$f(u,v) = \dfrac{u}{v}$,   $u _{0} = v _{0} = 1 , \; v _{0} = v _{1} = 2$.
Then 
$f(1,1)=1$, $f ( 2 , 2 ) = 2$, and
$f ( u , v ) \in [ \frac{1}{2} , 2 ] =: D$.
We have 
$\bigcup_{\theta \in [0,1]}
\theta f ( 1 , 1 ) + ( 1 - \theta ) f ( 2, 2 )
\in [ 1 , 2 ]$ being a proper subset of $D$.
But, 
$\underline{f} = \frac{1}{2}$, $\overline{f} = 2$, and
$\bigcup_{\theta \in [0,1]}
\theta \underline{f} + (1 - \theta) \overline{f} = [ 1/2 , 2 ] \equiv D$.
\end{remark}

\begin{theorem} \label{t-2}
There is a non-trivial time interval $[0,T_{1}]$ where the limiting solution 
$U=\lim_{\Delta \to 0}U_{\Delta}$ from Theorem \ref{lw} is the one 
satisfying Definition \ref{def-dafermos} among all classical solutions with values 
in the bounded set $B_{T}$.
\end{theorem}

\begin{proof}

Suppose that $U$ is an MES solution to (\ref{claws}), that is
a weak solution obtained as the limit of 
the approximations $U_{\Delta}$ that exist because of Theorem \ref{lw}. 
Suppose $\widehat{U}$ is a classical solution to the system (\ref{claws}) 
that violates the conditions of the Definition \ref{def-dafermos} 
after a time $t=T_{0}$. That is, $\mathcal{E}(U)(T_{0})
=\mathcal{E}(\widehat{U})(T_{0})$ and for every sequence
$(T_{k})_{k\in \mathbb{N}}$  converging to $T_{0}$, $T_{k}>T_{0}$,
\[
\begin{split}
& \int_{-L}^{L} \eta(U(x,t)) dx \leq \int_{-L}^{L} \eta(\widehat{U}(x,t))dx,
\; t\leq T_{0} \text{ and }\\ 
& \int_{-L}^{L} \eta(U(x,T_{k}))dx - \int_{-L}^{L} 
\eta(\widehat{U}(x,T_{k}))dx >\mu>0.
\end{split}
\]
Like in Lemma \ref{lemma3.1}, using the integral form (\ref{if}), we approximate $\hat{U}$ by taking
\begin{equation} \label{hatscheme}
\widehat{U}_{i}^{j+1} = \widehat{U}_{i}^{j} 
-\dfrac{\Delta t}{\Delta x}(\widehat{F}_{i+1/2}^{j}-\widehat{F}_{i-1/2}^{j}).
\end{equation}

Choose $T_{k}$ such that $\Delta t=T_{k}-T_{0}<\mu$, 
One can make the approximation arbitrary fine, so we may take $\Delta t$ small enough 
such that 
$$
\Big| \int_{-L}^{L}\eta(\hat{U}(x,T_{k}))dx-\sum_{i\in I}\eta(\hat{U}_{i}^{j})\Delta x
\Big| < \varepsilon, \; t_{j}=T_{k}.
$$
when $|\Delta|\to 0$, which determines the solution $U$.
By (\ref{hatlinf}), for 
$\varepsilon$ small enough, we may change the definition of $\widehat{U}_{i}^{j}$,
$i\in I$, $j=1,2, \ldots$ for the cells with indices belonging to 
$I \setminus I_{\rm nd}^{j}\cup I_{\rm nd}^{j+1}\cup I_{\rm m}^{j}$ and obtain the approximation
$\{\breve{U}_{i}^{j}, \; i\in I, \; j=1,2, \ldots\}$ so that
the flux function in (\ref{hatscheme}) satisfies the condition (\ref{nflux}).
That is, $(\breve{U}_{i}^{j})_{i\in I}$ is obtained by the scheme from 
Definition \ref{defscheme} for some choice of $\theta_{i}^{j}\in [0,1]^{m}$, 
$i\in I$, $j\geq 0$. So, for the same data at $t=T_{k}=$, 
$\eta^{j+1}(U_{\Delta})\leq \eta^{j+1}(\hat{U})$, where we have used the notation
$\eta^{j}(V):=\sum_{i\in I} \eta(V_{i}^{j})$ for a piecewise constant
approximation $V_{i}^{j}$, $i\in I$, $j=1,2, \ldots$ of the function $V$.

At the same time, using the differentiability of $\eta$, 
boundedness od $\widehat{U}$, and (\ref{hatlinf}), we have the following
inequality
\[
\begin{split}
& \Big| \int_{\Omega_{T}} \Big(\eta(\widehat{U}(\cdot,t_{j+1}))
- \eta^{j+1}(\widetilde{U})\Big) dx \Big| 
= \Big| \int_{\Omega_{T}}\eta(\widehat{U}(\cdot,t_{j+1}))dx 
- \sum_{i\in I}\eta(\widehat{U}_{i}^{j+1})\Delta x \Big| \\
\leq & \|\nabla \eta (\widehat{U})\|_{L^{\infty}(D_{T})}
\sum_{i\in I} \|\widehat{U}(x,t_{j})-\widehat{U}_{i}^{j}\|_{L^{\infty}(D_{T})} \Delta x
< L_{B}\varepsilon,
\end{split}
\]  
for $\Delta t<\mu$, where $L_{B}:=\|\nabla \eta (\widehat{U})\|_{L^{\infty}(D_{T})}
<\infty$.

That contradicts the fact that $\int_{-L}^{L} \eta(U(\cdot,T_{k})) 
- \int_{-L}^{L} \eta(\widehat{U}(\cdot,T_{k}))>\mu$ if $U_{\Delta}=\widetilde{U}$ in $t=T_{0}$.
\end{proof}

\begin{remark}
Let us note that Theorem \ref{t-2} holds only for one-dimensional 
conservation law systems. Moreover, there is no general answer about 
singularities, even for 1D systems, to the best of our knowledge. 
Therefore, we compare the MES solutions with the ones obtained by the most 
efficient algorithms for conservation law systems and the results 
from \cite{Bressan} and \cite{Diperna}.
The situation is even more complicated in the multidimensional case. 
There are no general efficient algorithms for solving these systems up to
our knowledge. 
In addition, nonuniqueness is a significant problem. One can look in 
\cite{CK2014}, \cite{DLS2010}, \cite{F2014} or \cite{FGSW}
to get some impression of problems with Dafermos' condition in 
multidimensional conservation law systems. It seems that each system 
has its peculiarities, and one has to look for the most physically 
reasonable admissibility conditions for the weak solutions.
Theorem \ref{lw} guarantees  an existence of MES regardless on the dimension. 
However, Theorem \ref{t-2} is not directly applicable, but  
it can be adapted to special cases: If a solution has singularities
contained in a set of Lebesgue measure zero and has a bounded variation
in each direction separately for $t$ fixed, then one can use multidimensional
rectangles and redefine the scheme (\ref{1dscheme}) adding $\theta$'s in
each variable. Nevertheless, computational effort is much higher in that
case and an efficient and precise numerical optimization scheme is needed.

\section{A numerical example}

Finding a minimum for a large number of variables is a very
challenging computational issues.
We present here a simple example computed using the Julia programming language
using the standard optimization packages.
They are based on the algorithm from \cite{WB}, the Interior
Point Newton method of second order. Let us remark that the first-order 
algorithms were less successful. The solution look almost the same, 
but with a lot of artificial
oscillations due to not-so-precise optimization results in a single time step.

The results shown below are for the problem
\[
\begin{split}
\partial_{t}\rho + \partial_{x}m = & 0 \\
\partial_{t}m + \partial_{x}\Big( \frac{m^{2}}{\rho}+\rho^{2}\Big) = & 0 	
\end{split}
\]  
with the Riemann initial data that produces 
a rarefaction wave followed by a shock.
The numerical results agree with the theoretical ones (wave speeds,
intermediate state between two waves). Calculations are made with 
1000 cells for the space interval, and 100 time steps (see Figure \ref{fig2}).
\end{remark}

\section*{Acknowledgment}
The author is grateful to Nata\v{s}a Kreji\'{c} for useful discussions
during the article preparation.

\begin{figure}[ht] \label{fig2}
	\includegraphics[width=0.7\textwidth]{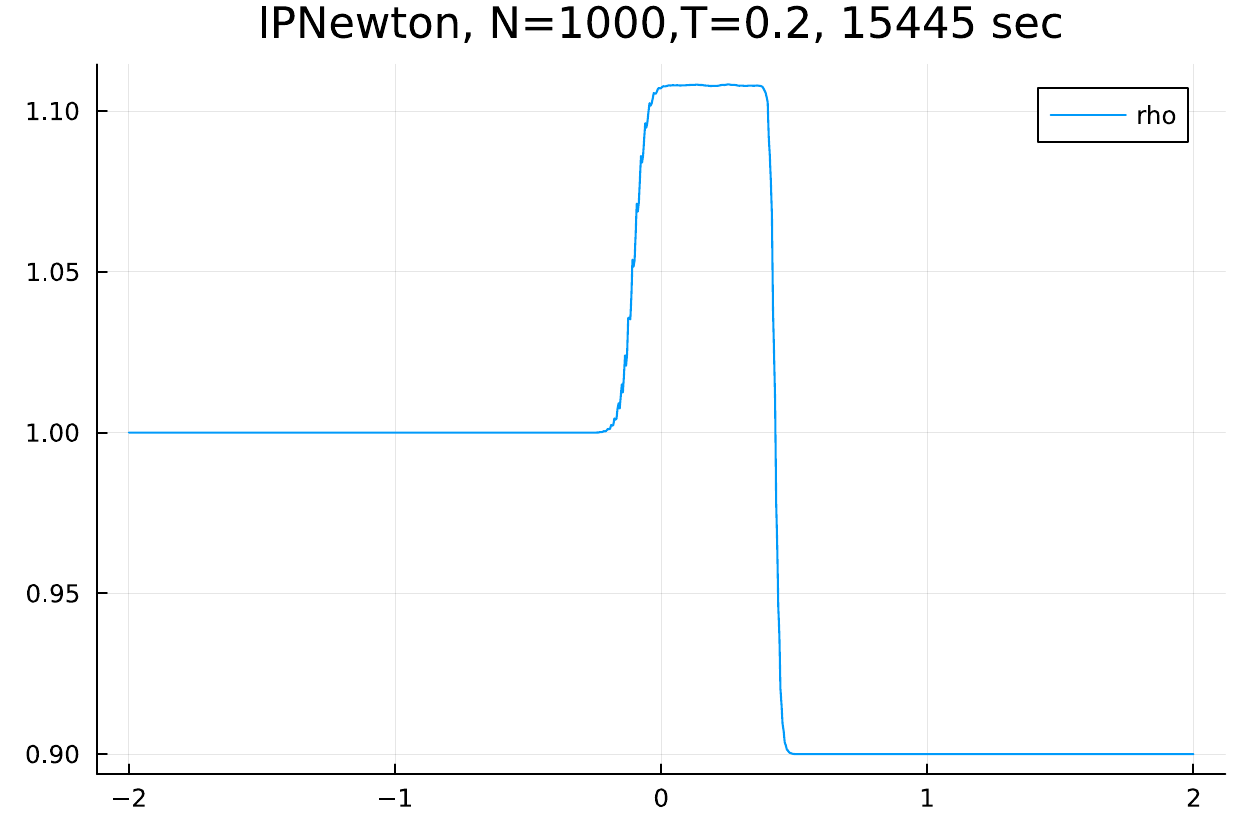}
    \includegraphics[width=0.7\textwidth]{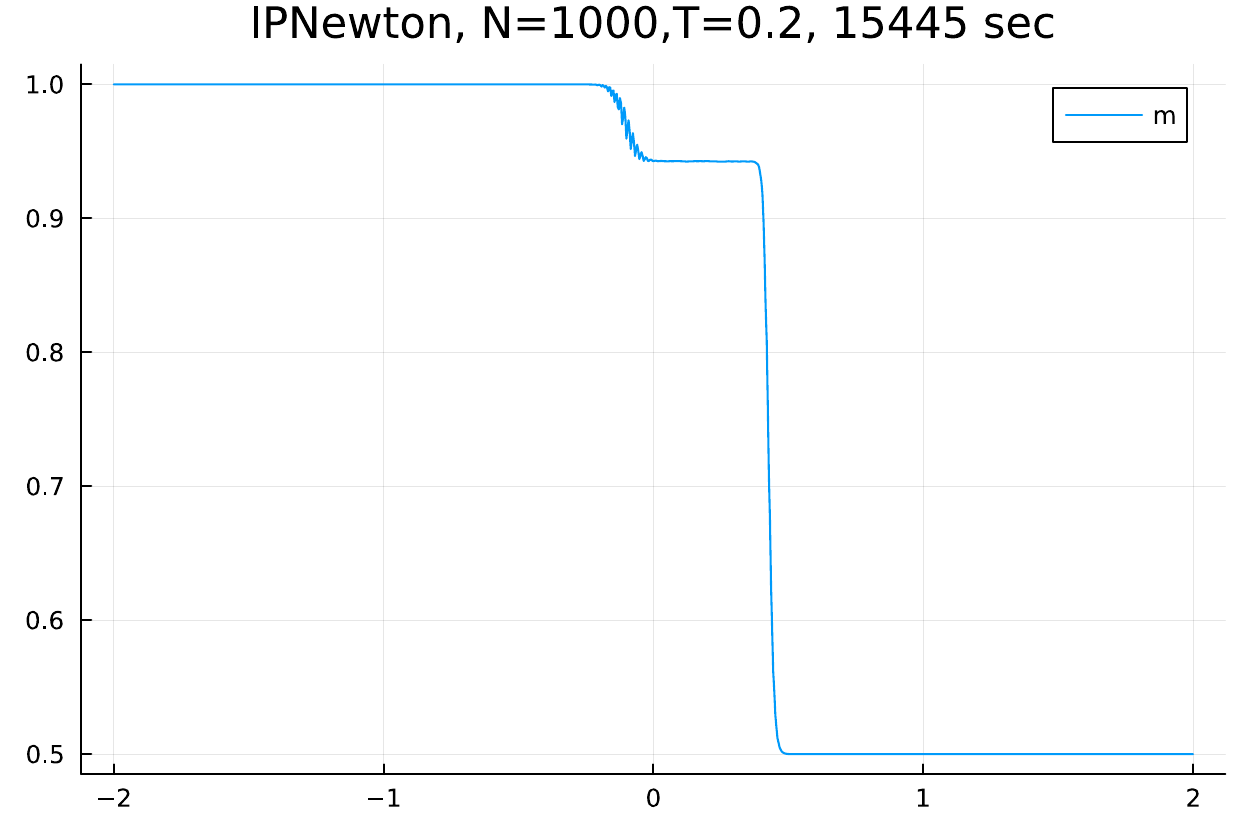}	
    \caption{$\rho$ and $m$ after 100 time steps}
\end{figure}

\end{document}